\documentclass{article}

\usepackage{amsmath} 

\usepackage{amssymb} 

\usepackage{amsthm} 

\usepackage{tikz} 

\usepackage{relsize} 

\usepackage{graphicx} 

\usepackage[pdfborder={0 0 0}]{hyperref} 

\usepackage[flushleft]{threeparttable}

\usepackage{thmtools}

\usepackage{placeins}

\usepackage{thm-restate}

\usepackage{tikz-cd} 

\usepackage{enumerate}

\usepackage{float}

\usepackage{scalerel}

\declaretheorem[name=Theorem,numberwithin=section]{thm}

\usepackage{xcolor,colortbl}

\usepackage{hyperref}

\newtheorem{theorem}[thm]{Theorem}

\newtheorem{example}[thm]{Example}

\theoremstyle{remark}

\theoremstyle{definition}

\begin{document}
\title{\textbf{The Number of Solutions to $ax+by+cz=n$ for Fibonacci and Lucas triplets}}
\author{\textbf{Pooja Teotia}}
\date{}

\maketitle
\begin{center}
    Department of Mathematics, Sant Longowal Institute Of Engineering And Technology - 148106 (Punjab) India \\
    poojateotia1117@gmail.com
\end{center}
\begin{abstract}
In this work we develop exact formulas to the number of solutions of $ax+by+cz=n$ in some special cases. In 2020, Binner gave a formula for the  number of non-negative integer solutions, $N(a,b,c;n)$ in non-negative integer pairs $(x,y,z),$ of the equation $ax+by+cz=n$ assuming that $a,b,c$ and $n$ are natural numbers. However, his formula was in summations of floor functions. Moreover,  he gave a reciprocity relation to solve these sums by generalizing the Gauss reciprocity relation; see \cite{Binner} for more details. Until now no exact formula has been found to solve these sums. We notice that these sums can be completely solved in some special cases, which lead us to find the number of solutions of the above equation in case of Fibonacci and Lucas triplets. In other words; If $a,b\ and\ c$ are chosen to be three consecutive Fibonacci or Lucas numbers then we determine the exact formula to the number of non-negative integer solutions $(x,y,z)$ of the equation \textbf{$F_ix+F_{i+1}y+F_{i+2}z=n$} and \textbf{$L_ix+L_{i+1}y+L_{i+2}z=n$} where i is fixed.
\end{abstract}

\section{Introduction}
Suppose we are given natural numbers $a$ and $b$ with $\gcd(a,b)=1.$ It is well-known that the equation $ax+by=n$ has solutions if and only if $\gcd(a,b)$ divides $n$. If $\gcd(a,b)$ does not divide $n$, then the above equation has no solutions. Therefore, there is no loss of generality in assuming that $\gcd(a,b)=1$.
Throughout, we will be concerned with nonnegative integer solutions only.
In $2000$, Tripathi \cite{AmiTri} discovered a formula to find the number of solutions of the equation $ax+by=n$. He used the method of generating functions, roots of unity, partial fractions and Taylor's series to obtain his formula.
Tripathi's result \cite{AmiTri} is given below.

\begin{theorem}[Tripathi ($2000$)]
The number of non-negative integer solutions of $ax+by=n$ is
$$ N(a,b;n) = \frac{n+aa'(n)+bb'(n)}{ab}-1.$$ 
where the quantities $a'(n)$ and $b'(n)$ are determined uniquely by
\begin{itemize}

\item $a'(n) \equiv -na^{-1}$ (mod $b$), $1 \leq a'(n) \leq b$.

\item $b'(n) \equiv -nb^{-1}$ (mod $a$), $1 \leq b'(n) \leq a$.

\end{itemize}
\end{theorem}

In 2020, Binner \cite{Binner} extended Tripathi's work and discovered a formula for the number of solutions $N(a,b,c;n)$ of the equation $ax+by+cz=n$ in non-negative integer triples $(x,y,z)$, though his result contains summations. 
\\
In this work, our task is to use Binner's \cite{Binner} result to find exact formulas for the number of solutions in the cases when the coefficients are consecutive Fibonacci or Lucas triplets. Let $F_i$ denote the $i^{th}$ Fibonacci number and $L_i$ denote the $i^{th}$ Lucas number. For given fixed natural numbers $i$ and $n$, we find an exact formula to the number of non-negative integer solutions, $N(F_i,F_{i+1},F_{i+2};n)$ and $N(L_i,L_{i+1},L_{i+2};n)$ of the equations $$F_ix+F_{i+1}y+F_{i+2}z=n$$ and $$L_ix+L_{i+1}y+L_{i+2}z=n$$ in Theorem \ref{FibFrob} and Theorem \ref{FibFrobb} respectively.
We will prove Theorem \ref{FibFrob} in Section \ref{section2} and Theorem \ref{FibFrobb} in Section \ref{Section3}. 

Next, we describe Binner's result for $N(a,b,c;n)$ in some detail.
Suppose that $a,$ $b,$ $c$ and $n$ are given natural numbers such that $\gcd(a,b,c)=1$. 
Binner defined a transformation that helped him to reduce the equation $ax+by+cz=n$ with pairwise coprime coefficients, i.e. $\gcd(a,b) = \gcd(b,c) = \gcd(a,c) = 1$. We refer the reader to \cite[Section 2.1]{Binner} to know more about this transformation method. He also introduced some notation as described below.
\begin{itemize}
	\item Define $ b'_1$ such that  $b'_1 \equiv - nb^{-1}$ (mod $a$)  with $ 1\leq b'_1 \leq a$. Moreover define $ c'_1$ such that $c'_1 \equiv bc^{-1}$ (mod $a$)  with $ 1\leq c'_1 \leq a$.
	\item Define  $ c'_2$ such that  $c'_2 \equiv - nc^{-1}$ (mod $b$)  with $ 1\leq c'_2 \leq b$. Moreover define $ a'_2$ such that $a'_2 \equiv ca^{-1}$ (mod $b$)  with $ 1\leq a'_2\leq b$.
	\item Define $ a'_3$ such that $a'_3 \equiv - na^{-1}$ (mod $c$)  with $ 1\leq a'_3 \leq c$. Moreover define  $b'_3$ such that $b'_3 \equiv ab^{-1}$ (mod $c$)  with $ 1\leq b'_3 \leq c$.
	\item $N_1 = n(n + a + b +c) + cbb'_1(a+1-c'_1(b'_1-1)) + acc'_2(b+1 - a'_2(c'_2-1))$ $+ baa'_3 (c+1-b'_3(a'_3-1)) $.
\end{itemize}

\begin{theorem}[Binner (2020)]
\label{MainThm}
\normalfont The number of non-negative integer solutions of $ax + by +cz = n$ is 
$$N(a,b,c;n) =   \frac{N_1}{2abc} + \sum_{i=1}^{b'_1-1} \Big\lfloor \frac{ic'_1}{a} \Big\rfloor + \sum_{i=1}^{c'_2-1} \Big\lfloor \frac{ia'_2}{b} \Big\rfloor + \sum_{i=1}^{a'_3-1} \Big\lfloor \frac{ib'_3}{c} \Big\rfloor  - 2 .$$
\end{theorem}

Clearly, Binner's formula involves summations of floor functions. There is no direct formula to find these summations. The best method to find these sums is a reciprocity relation given by Binner as described below.

 \begin{theorem}[Binner ($2020$)]
 \label{GenReci}
 \normalfont Let $a$, $b$, $c$, and $K$ be positive integers such that $b < a$, $c < a$, $\gcd(a,c) = 1$, and $K = \left\lfloor \frac{bc}{a} \right\rfloor$. Then $$\sum_{i=1}^{b} \left\lfloor \frac{ic}{a} \right\rfloor + \sum_{i=1}^{K} \left\lfloor \frac{ia}{c} \right\rfloor = bK.$$ 
\end{theorem}

We noticed some special families of $(a,b,c)$ for which we can determine an exact formula to find the number of solutions. 
From above notation, we have $c'_1 \equiv bc^{-1}$ mod $a$. This suggests that $a,\ b,\ and\ c$ should be chosen such that the inverse of one number can be easily found modulo the others.  We realized that consecutive Fibonacci and Lucas triples are such families of triplets.

The next theorem describes our main result, which provides an exact answer to the number of solutions of the equation $$F_ix+F_{i+1}y+F_{i+2}z=n.$$ 
We need to introduce some notation.
\begin{itemize}
    \item $B'_1$ is defined such that $B'_1 \equiv (-1)^i nF_{i-2}$(mod $F_i)$, and $1 \leq B'_1 \leq F_i$.
    \item $C'_2$ is defined such that $C'_2 \equiv (-1)^i nF_{i}$ (mod $F_{i+1})$, and $1 \leq C'_2 \leq F_{i+1}$.
    \item $A'_3$ is defined such that $A'_3 \equiv (-1)^i nF_{i+1}$ (mod $F_{i+2})$, and $1 \leq A'_3 \leq F_{i+2}$.
    \item $N_{2}$ is defined such that $N_2 = n(n+F_i+F_{i+1}+F_{i+2})+F_{i+2}F_{i+1}B'_1(F_i+1-(B'_1-1))+F_iF_{i+2}C'_2(F_{i+1}+1-(C'_2-1))+F_{i+1}F_iA'_3(F_{i+2}+1-(F_{i+2}-1)(A'_3-1))$.
\end{itemize}
\begin{theorem}
\label{FibFrob}
 For a fixed natural number $i,\ F_i,\ F_{i+1}\  and \ F_{i+2}$ denote the $i^{th},\  (i+1)^{th}\  and\  (i+2)^{th}$ Fibonacci numbers. Then, with the above notation, the number of solutions of the equation $ F_i x + F_{i+1}y + F_{i+2} z = n $ is given by $$N(F_i,F_{i+1},F_{i+2};n)=\frac{N_2}{2F_iF_{i+1}F_{i+2}}+\frac{(A'_3-1)(A'_3-2)}{2} - 2.$$
\end{theorem}

\section{Formula for $N(F_i, F_{i+1},F_{i+2};n)$}
\label{section2}

In this section, we prove our exact formula for the number of solutions, $N(F_i,F_{i+1},F_{i+2})$ of the equation $F_ix+F_{i+1}y+F_{i+2}z=n$ where $i$ and $n$ are given natural numbers. Clearly, $F_i,F_{i+1}\ and\ F_{i+2}$ are pairwise coprime numbers.

\label{StatProofch3th}

The next theorem describes our main result, which provides an exact answer to the number of solutions of the equation $F_ix+F_{i+1}y+F_{i+2}z=n$.\\
Before that, we need to introduce some notations.
\begin{itemize}
    \item $B'_1$ is defined such that $B'_1 \equiv (-1)^i nF_{i-2}(modF_i)$, and $1 \leq B'_1 \leq a$.
    \item $C'_2$ is defined such that $C'_2 \equiv (-1)^i nF_{i}(modF_{i+1})$, and $1 \leq C'_2 \leq b$.
    \item $A'_3$ is defined such that $A'_3 \equiv (-1)^i nF_{i+1}(modF_{i+2})$, and $1 \leq A'_3 \leq c$.
    \item $N_{2}$ is defined such that $N_2 = n(n+F_i+F_{i+1}+F_{i+2})+F_{i+2}F_{i+1}B'_1(F_i+1-(B'_1-1))+F_iF_{i+2}C'_2(F_{i+1}+1-(C'_2-1))+F_{i+1}F_iA'_3(F_{i+2}+1-(F_{i+2}-1)(A'_3-1))$.
\end{itemize}
\begin{theorem}
\label{FibFrob}
\normalfont For a fixed natural number $i,\ F_i,\ F_{i+1}\  and \ F_{i+2}$ denote the $i^{th},\  (i+1)^{th}\  and\  (i+2)^{th}$ Fibonacci numbers. Then, with the above notations, the number of solutions of the equation $ F_i x + F_{i+1}y + F_{i+2} z = n $ is given by $$N(F_i,F_{i+1},F_{i+2};n)=\frac{N_2}{2F_iF_{i+1}F_{i+2}}+\frac{(A'_3-1)(A'_3-2)}{2} - 2.$$
\end{theorem}
\textbf{Proof.}
By Binner's result Theorem \ref{MainThm}, we have
\begin{equation}
\label{General}
N(a,b,c;n) = \frac{N_1}{2abc} + \sum_{i=1}^{b'_1-1} \left\lfloor \frac{ic'_1}{a} \right\rfloor + \sum_{i=1}^{c'_2-1} \left\lfloor \frac{ia'_2}{b} \right\rfloor + \sum_{i=1}^{a'_3-1} \left\lfloor \frac{ib'_3}{c} \right\rfloor - 2 .
\end{equation}

We begin by computing the quantities $b'_1,c'_2,a'_3$, and $c'_1,a'_2,b'_3$ in the case when $a=F_i,$ $b=F_{i+1}$ and $c=F_{i+2}$.
We recall these quantities from section 1.
\begin{itemize}

	\item Define $ b'_1$ such that $b'_1 \equiv - nb^{-1}$ (mod $a$) with $ 1\leq b'_1 \leq a$. Moreover, define $ c'_1$ such that $c'_1 \equiv bc^{-1}$ (mod $a$) with $ 1\leq c'_1 \leq a$.

	\item Define $ c'_2$ such that $c'_2 \equiv - nc^{-1}$ (mod $b$) with $ 1\leq c'_2 \leq b$. Moreover, define $ a'_2$ such that $a'_2 \equiv ca^{-1}$ (mod $b$) with $ 1\leq a'_2\leq b$.

	\item Define $ a'_3$ such that $a'_3 \equiv - na^{-1}$ (mod $c$) with $ 1\leq a'_3 \leq c$. Moreover, define $b'_3$ such that $b'_3 \equiv ab^{-1}$ (mod $c$) with $ 1\leq b'_3 \leq c$.

	\item Define $N_1 = n(n + a + b +c) + cbb'_1(a+1-c'_1(b'_1-1)) + acc'_2(b+1 - a'_2(c'_2-1))$ $+ baa'_3 (c+1-b'_3(a'_3-1)) $.

\end{itemize}
Firstly, we will compute $b'_1$ and prove that $b'_1 = B'_1$. In our case, we have $b'_1 \equiv -nF_{i+1}^{-1}(modF_i)$. Therefore,  to obtain $b'_1$ , we first need to find the 
modular inverse of $F_{i+1}$ with respect to $F_i$. The above quantities are greatly simplified using the properties of Fibonacci numbers, especially the Cassini's identity, namely
$$ F_{i+1}F_{i-1} = F_i^2 + (-1)^i.$$
Reducing modulo $F_i$ and
multiplying by $(-1)^i$ on both sides in the above equation, we get
$$F_{i+1}((-1)^iF_{i-1})\equiv 1 \bmod F_i.$$
\begin{equation}
    \label{Inverse}
    F_{i+1}^{-1} \equiv (-1)^iF_{i-1} \hspace{.2cm} (\bmod \hspace{.2cm}F_i).
\end{equation}
Therefore,
\begin{align*}
b'_1 &\equiv -nF_{i+1}^{-1} \hspace{.2cm}(\bmod \hspace{.2cm}F_i) \\
&\equiv -n (-1)^iF_{i-1} \hspace{.2cm} (\bmod \hspace{.2cm}F_i) \\
&\equiv (-1)^i n (-F_{i-1}) \hspace{.2cm} (\bmod \hspace{.2cm}F_i) \\
&\equiv (-1)^i n (F_i-F_{i-1}) \hspace{.2cm} (\bmod \hspace{.2cm}F_i) \\
&\equiv (-1)^i n F_{i-2} \hspace{.2cm} (\bmod \hspace{.2cm}F_i) \\
&\equiv B'_1 \hspace{.2cm}(\bmod \hspace{.2cm}F_i). 
\end{align*}
It is easy to see that $1 \leq b'_1, B'_1 \leq F_i$. Thus, $b'_1 = B'_1$. Next, we prove that $c'_2 = C'_2$. 
\\
\\
We have $c'_2 \equiv -nF_{i+2}^{-1}(modF_{i+1})$.  Therefore to calculate $c'_2$, we need to find the modular inverse of $F_{i+2}$ with respect to $F_{i+1}$. Replacing
$i$ with $i+1$ in \eqref{Inverse}, we get
$$F_{i+2}^{-1} \equiv (-1)^{i+1}F_{i} \hspace{.2cm} (\bmod \hspace{.2cm}F_{i+1}).$$
Hence, 
\begin{align*}
c'_2 &\equiv -nF_{i+2}^{-1} \hspace{.2cm}(\bmod \hspace{.2cm}F_{i+1}) \\
&\equiv -n (-1)^{i+1}F_{i} \hspace{.2cm} (\bmod \hspace{.2cm}F_{i+1}) \\
&\equiv (-1)^i n F_{i}) \hspace{.2cm} (\bmod \hspace{.2cm}F_{i+1}) \\
&\equiv C'_2 \hspace{.2cm}(\bmod \hspace{.2cm}F_{i+1}). 
\end{align*}
It is also easy to see that $1 \leq c'_2, C'_2 \leq F_{i+1}$. Thus, $c'_2 = C'_2$. Now, we prove that $a'_3 = A'_3$.

We have
\begin{align}
\label{Ii}
a'_3 &\equiv -nF_{i}^{-1} \hspace{.2cm}(\bmod \hspace{.2cm}F_{i+2}) \notag\\
&\equiv -n (F_{i+2}-F_{i+1})^{-1} \hspace{.2cm} (\bmod \hspace{.2cm}F_{i+2}) \notag\\
&\equiv -n (-F_{i+1})^{-1} \hspace{.2cm} (\bmod \hspace{.2cm}F_{i+2}) \notag\\
&\equiv n F_{i+1}^{-1} \hspace{.2cm} (\bmod \hspace{.2cm}F_{i+2}) 
\end{align}

From Cassini's identity, we have $F_{i+3}F_{i+1} = F_{i+2}^2 + (-1)^i.$
Again, reducing modulo $F_{i+2}$, and multiplying by $(-1)^i$ on both sides of this equation, we have $F_{i+3}F_{i+1} \equiv (-1)^i (\bmod F_{i+2}).$ which gives $ F_{i+1} \left((-1)^i F_{i+3}\right) \equiv 1 \bmod F_{i+2}.$ Therefore, 
\begin{align}
\label{IIii}
F_{i+1}^{-1} &\equiv (-1)^i F_{i+3} \hspace{.2cm} (\bmod \hspace{.2cm} F_{i+2}) \notag\\ 
& \equiv (-1)^i \left(F_{i+2} + F_{i+1}\right) \hspace{.2cm} (\bmod \hspace{.2cm} F_{i+2}) \notag\\
& \equiv (-1)^i F_{i+1} \hspace{.2cm} (\bmod \hspace{.2cm} F_{i+2}).
\end{align}

From equations \eqref{Ii} and \eqref{IIii}, we obtain

\begin{align*}
    a'_3 &\equiv (-1)^i n F_{i+1} \hspace{.2cm} (\bmod \hspace{.2cm} F_{i+2}) \\
    &\equiv A'_3 \hspace{.2cm} (\bmod \hspace{.2cm} F_{i+2}). 
\end{align*}

Therefore, $a'_3 \equiv A'_3$ (mod $F_{i+2}$), and $1 \leq a'_3, A'_3 \leq F_{i+2}$. 

Till now, we have simplified the quantities $b'_1 = B'_1$, $c'_2 = C'_2$ and $a'_3 = A'_3$. Next, we simplify the quantities $c'_1$, $a'_2$ and $b'_3$.
In our case, we have
\begin{align*}
c'_1 &\equiv bc^{-1} \hspace{.2cm} (\bmod \hspace{.2cm}a) \\
&\equiv F_{i+1}F_{i+2}^{-1} \hspace{.2cm}(\bmod \hspace{.2cm}F_i) \\
&\equiv (F_{i+2}-F_i) F_{i+2}^{-1} \hspace{.2cm} (\bmod \hspace{.2cm}F_i) \\
&\equiv F_{i+2} F_{i+2}^{-1} \hspace{.2cm} (\bmod \hspace{.2cm}F_i) \\
&\equiv 1 \hspace{.2cm} (\bmod \hspace{.2cm}F_i). 
\end{align*}

Thus, $c'_1 \equiv 1$ mod $F_i$, and $1 \leq c'_1 \leq F_i$. That is, $c'_1 = 1$.
Similarly, 

\begin{align*}
a'_2 &\equiv ca^{-1} \hspace{.2cm} (\bmod \hspace{.2cm}b) \\
&\equiv F_{i+2}F_{i}^{-1} \hspace{.2cm} (\bmod \hspace{.2cm}F_{i+1}) \\
&\equiv (F_{i+1}+F_i) F_{i}^{-1} \hspace{.2cm} (\bmod \hspace{.2cm}F_{i+1}) \\
&\equiv F_{i} F_{i}^{-1} \hspace{.2cm} (\bmod \hspace{.2cm} F_{i+1}) \\
&\equiv 1 \hspace{.2cm} (\bmod \hspace{.2cm}F_{i+1}). 
\end{align*}

Thus, $a'_2 \equiv 1$ mod $F_{i+1}$, and $1 \leq a'_2 \leq F_{i+1}$. That is, $a'_2 = 1$. Similarly, 

\begin{align*}
b'_3 &\equiv ab^{-1} \hspace{.2cm} (\bmod \hspace{.2cm}c) \\
&\equiv F_{i}F_{i+1}^{-1} \hspace{.2cm} (\bmod \hspace{.2cm}F_{i+2}) \\
&\equiv (F_{i+2}-F_{i+1}) F_{i}^{-1} \hspace{.2cm} (\bmod \hspace{.2cm}F_{i+2}) \\
&\equiv -F_{i+1} F_{i+1}^{-1} \hspace{.2cm} (\bmod \hspace{.2cm}F_{i+2}) \\
&\equiv -1 \hspace{.2cm} (\bmod \hspace{.2cm}F_{i+2}). 
\end{align*}

Thus, $b'_3 \equiv -1$ mod $F_{i+2}$, and $1 \leq b'_3 \leq F_{i+2}$. That is, $b'_3 = F_{i+2} - 1$. 
Putting the values of $b'_1, c'_2, a'_3$, and of $c'_1, a'_2, b'_3$ in the expression for $N_1$, we get $N_1 = N_2$. Finally, we solve the summations involved in Binner's formula. Using, $b'_1 \leq F_i$, $c'_2 \leq F_{i+1}$, and $a'_3 \leq F_{i+2}$, we have

$$\sum_{i=1}^{b'_1-1} \left\lfloor \frac{ic'_1}{a} \right\rfloor=\sum_{i=1}^{b'_1-1} \left\lfloor \frac{ic'_1}{F_i} \right\rfloor =\sum_{i=1}^{b'_1-1} \left\lfloor \frac{i}{F_i} \right\rfloor=0,$$\\
$$\sum_{i=1}^{c'_2-1} \left\lfloor \frac{ia'_2}{b} \right\rfloor=\sum_{i=1}^{c'_2-1} \left\lfloor \frac{ia'_2}{F_{i+1}} \right\rfloor=\sum_{i=1}^{c'_2-1} \left\lfloor \frac{i}{F_{i+1}} \right\rfloor=0, $$

\begin{align*}
\sum_{j=1}^{a'_3-1} \left\lfloor \frac{jb'_3}{c} \right\rfloor &=\sum_{j=1}^{a'_3-1} \left\lfloor \frac{j(F_{i+2}-1)}{F_{i+2}} \right\rfloor\\
&=\sum_{j=1}^{a'_3-1} \left\lfloor j- \frac{j}{F_{i+2}} \right\rfloor\\     
&=\sum_{j=1}^{a'_3-1}(j-1)\\
&=\frac{(a'_3-2)(a'_3-1)}{2}.
\end{align*}

Putting all the values of these summations back in \eqref{General} completes the proof of our Theorem \ref{FibFrob}. 

Thus, we have an exact formula for finding the number of solutions of the equation $ F_i x + F_{i+1}y + F_{i+2} z = n $ for given natural numbers $i$ and $n$. Let us apply our formula for an example.
\begin{example}
\normalfont

Suppose $i=12$. Then, $F_{i} = 144$, $F_{i+1} = 233$, and $F_{i+2} = 377$. Consider the equation $144x+233y+377z=425896$.
Since $i = 12$ is even, thus $(-1)^i = 1$. Therefore, we have
\begin{align*}
B'_1&\equiv nF_{i-2}\hspace{.2cm}(\bmod \hspace{.2cm}F_i)\\
&\equiv 425896\times 55 \hspace{.2cm}(\bmod \hspace{.2cm}144)\\
&\equiv 88 \hspace{.2cm} (\bmod \hspace{.2cm}144).
\end{align*}

Clearly, $1 \leq B'_1 \leq 143$. Thus, $B'_1 = 88$. Similarly,
\begin{align*}
C'_2&\equiv nF_i \hspace{.2cm} (\bmod \hspace{.2cm}F_{i+1})\\
&\equiv 425896\times 144 \hspace{.2cm}(\bmod \hspace{.2cm}233)\\
&\equiv 162 \hspace{.2cm} (\bmod \hspace{.2cm}233)\\
\end{align*}
Clearly, $1 \leq C'_2 \leq 232$. Thus, $C'_2 = 162$. Similarly,
\begin{align*}
A'_3 &\equiv nF_{i+1} \hspace{.2cm} (\bmod \hspace{.2cm}F_{i+2})\\
&\equiv 425896\times 233 \hspace{.2cm} (\bmod \hspace{.2cm}377)\\
&\equiv 205 \hspace{.2cm} (\bmod \hspace{.2cm}377)\\
\end{align*}
Clearly, $1 \leq C'_2 \leq 376$. Thus, $A'_3 = 205$.
\\
\\
We have,
$N_2 = n(n+F_i+F_{i+1}+F_{i+2})+F_{i+2}F_{i+1}B'_1(F_i+1-(B'_1-1))+F_iF_{i+2}C'_2(F_{i+1}+1-(C'_2-1))+F_{i+1}F_iA'_3(F_{i+2}+1-(F_{i+2}-1)(A'_3-1))$.
\\ 
\\
Putting the values of $n,F_i,F_{i+1},F_{i+2},B'_1,C'_2$, and $A'_3$ in the expression of $N_2$, we find that $$N_2 = -342183561408.$$ Then, Substituting the values of $N_2$ and $A'_3$ in Theorem \ref{FibFrob}, we get $$N(144,233,377; 425896) = 7178.$$
Hence, The total number of non-negative integer solutions of $144x+233y+377z=425896$ are 7178.
\end{example}

\section{Formula for $N(L_i,L_{i+1},L_{i+2};n)$}
\label{Section3}
Now we prove our next exact formula for the number of non-negative integer solutions $N(L_i,L_{i+1},L_{i+2})$ of the equation  $L_ix+L_{i+1}y+L_{i+2}z=n$.
Let us introduce some notation.
\begin{itemize}
    \item $B''_1$ is defined such that $B''_1 \equiv (-1)^{i+1} nL_{i-2}5^{-1}(modL_i)$, and $1 \leq B''_1 \leq a$.
    \item $C''_2$ is defined such that $C''_2 \equiv (-1)^{i+1} nL_{i}5^{-1}(modL_{i+1})$, and $1 \leq C''_2 \leq b$.
    \item $A''_3$ is defined such that $A''_3 \equiv (-1)^{i+1}nL_{i+1}5^{-1}(modL_{i+2})$, and $1 \leq A''_3 \leq c$.
    \item $N_{3}$ is defined such that $N_3 = n(n+L_i+L_{i+1}+L_{i+2})+L_{i+2}L_{i+1}B''_1(L_i+1-(B''_1-1))+L_iL_{i+2}C''_2(L_{i+1}+1-(C''_2-1))+L_{i+1}L_iA''_3(L_{i+2}+1-(L_{i+2}-1)(A''_3-1))$.
\end{itemize}

\begin{theorem}
\label{FibFrobb}
\normalfont For a fixed natural number $i,\ L_i,\ L_{i+1}\ and \ L_{i+2}$ denote the $i^{th},\  (i+1)^{th}\  and\ (i+2)^{th}$ Lucas numbers. Then, with the above notations, the number of solutions of the equation $ L_i x + L_{i+1}y + L_{i+2} z = n $ is given by $$N(L_i,L_{i+1},L_{i+2};n)=\frac{N_3}{2L_iL_{i+1}L_{i+2}}+\frac{(A''_3-1)(A''_3-2)}{2} - 2.$$
\end{theorem}
\textbf{Proof.}
By Binner's formula Theorem \ref{MainThm}, we have
\begin{equation}
\label{Generall}
N(a,b,c;n) = \frac{N_1}{2abc} + \sum_{i=1}^{b'_1-1} \left\lfloor \frac{ic'_1}{a} \right\rfloor + \sum_{i=1}^{c'_2-1} \left\lfloor \frac{ia'_2}{b} \right\rfloor + \sum_{i=1}^{a'_3-1} \left\lfloor \frac{ib'_3}{c} \right\rfloor - 2 .
\end{equation}
When $a,$ $b,$ and $c$ are chosen as $L_i,$ $L_{i+1},$ and $L_{i+2}$ respectively then the quantities $c'_1,\ a'_2\ and\ b'_3$ can be easily solved using the Cassini’s Identity for Lucas numbers, namely
\begin{equation}
\label{CI}
L_i^2-L_{i-1}L_{i+1}=(-1)^i5
\end{equation}
In our case, we have $b'_1 \equiv -nL_{i+1}^{-1}(modL_i)$.  Therefore, to calculate $b'_1$ we need to evaluate modular inverse of $L_{i+1}$ with respect to $L_i$.
Replacing i with i+1 in \eqref{CI}, we get
$$L_{i+1}^2-L_iL_{i+2}=(-1)^{i+1}5$$
Reducing the above equation modulo $L_i$ and 
then multiplying both sides with $(-1)^{i+1}5^{-1}$, we obtain $$L_{i+1} \left((-1)^{i+1}L_{i+1}\ 5^{-1}\right) \equiv 1 \bmod L_i.$$ 

From here, we get
    \label{Inversee}
    \begin{align*}
L_{i+1}^{-1} &\equiv (-1)^{i+1}L_{i+1}\ 5^{-1} \hspace{.2cm} (\bmod \hspace{.2cm}L_i).\\
&\equiv (-1)^{i+1}\left(L_{i}+L_{i-1}\right)\ 5^{-1} \hspace{.2cm} (\bmod \hspace{.2cm}L_i).\\
&\equiv (-1)^{i+1}L_{i-1}\ 5^{-1} \hspace{.2cm} (\bmod \hspace{.2cm}L_i).
    \end{align*}
so we need the $5^{-1}(\bmod \hspace{.2cm} L_i)$\\
Considering first few Lucas numbers\\
$$L_0=2, L_1=1, L_2=3, L_3=4, L_4=7,L_5=11, L_6=18, L_7=29, L_8=47 \ldots$$
Clearly,\\
$$L_0\equiv 2(\bmod \hspace{.2cm} 5), L_1\equiv 1(\bmod \hspace{.2cm} 5), L_2\equiv 3(\bmod \hspace{.2cm} 5), L_3\equiv 4(\bmod \hspace{.2cm} 5)$$
$$L_4\equiv 2(\bmod \hspace{.2cm} 5), L_5\equiv 1(\bmod \hspace{.2cm} 5), L_6\equiv 3(\bmod \hspace{.2cm} 5), L_7\equiv 4(\bmod \hspace{.2cm} 5)$$\\
Note that $L_{4k}\equiv 2(\bmod \hspace{.2cm} 5),$ $L_{4k+1}\equiv 1(\bmod \hspace{.2cm} 5)$ and $L_{4k+2}\equiv 3(\bmod \hspace{.2cm} 5).$Therefore,
$L_{4k}^{-1}
\equiv 3(\bmod \hspace{.2cm}5),$ $L_{4k+1}^{-1}
\equiv 1(\bmod \hspace{.2cm}5)$ and $L_{4k+2}^{-1}
\equiv 2(\bmod \hspace{.2cm}5)$ respectively.
When i=4k,
$$L_{4k}^{-1}\equiv 3(\bmod \hspace{.2cm} 5)$$
$$3L_{4k}\equiv 1(\bmod \hspace{.2cm} 5)$$
$$3L_{4k}=1+5u$$
$$5(-u)\equiv 1(\bmod \hspace{.2cm} L_{4k})$$
$$5^{-1}\equiv \frac{1-3L_{4k}}{5}(mod \hspace{.2cm} L_{4k})$$
Similarly, we also have
$$5^{-1}=\left\{\begin{array}{llll}
    \frac{1-L_{4k+1}}{5}(\bmod L_{4k+1}) ;&i=4k+1\\
    \\
    \frac{1-2L_{4k+2}}{5}(\bmod L_{4k+2}) ;&i=4k+2\\
    \\
    \frac{1-4L_{4k+3}}{5}(\bmod L_{4k+3}) ;&i=4k+3
    \end{array}\right.
    $$
Firstly, we will compute the quantities $b'_1$, $c'_2$ and $a'_3$.
\begin{align*}
b'_1 &\equiv -nL_{i+1}^{-1} \hspace{.2cm}(\bmod \hspace{.2cm}L_i) \\
&\equiv -n (-1)^{i+1}L_{i-1}\  5^{-1} \hspace{.2cm} (\bmod \hspace{.2cm}L_i) \\
&\equiv (-1)^{i+1} n (-L_{i-1})\  5^{-1} \hspace{.2cm} (\bmod \hspace{.2cm}L_i) \\
&\equiv (-1)^{i+1} n (L_i-L_{i-1})\  5^{-1} \hspace{.2cm} (\bmod \hspace{.2cm}L_i) \\
&\equiv (-1)^{i+1} n L_{i-2}\  5^{-1} \hspace{.2cm} (\bmod \hspace{.2cm}L_i) \\
&\equiv B''_1 \hspace{.2cm}(\bmod \hspace{.2cm}L_i). 
\end{align*}

Note that $c'_2 \equiv -nL_{i+2}^{-1}(\bmod L_{i+1})$. To find $c'_2$, we first need modular inverse of $L_{i+2}$ with respect to $L_{i+1}$. Replacing $i$ with $i+2$ in equation\eqref{CI}, we obtain
$$L_{i+2}^{-1} \equiv (-1)^{i+2}L_{i}\ 5^{-1}\hspace{.2cm} (\bmod \hspace{.2cm}L_{i+1}).$$
Therefore, 
\begin{align*}
c'_2 &\equiv -nL_{i+2}^{-1} \hspace{.2cm}(\bmod \hspace{.2cm}L_{i+1}) \\
&\equiv -n (-1)^{i+2}L_{i}\ 5^{-1} \hspace{.2cm} (\bmod \hspace{.2cm}L_{i+1}) \\
&\equiv (-1)^{i+3} n L_{i}\ 5^{-1} \hspace{.2cm} (\bmod \hspace{.2cm}L_{i+1}) \\
&\equiv (-1)^{i+1} n L_{i}\ 5^{-1} \hspace{.2cm} (\bmod \hspace{.2cm}L_{i+1}) \\
&\equiv C''_2 \hspace{.2cm}(\bmod \hspace{.2cm}L_{i+1}). 
\end{align*}

Note that $a'_3 \equiv -nL_{i}^{-1}(modL_{i+2})$.  To find $a'_3$, we first need modular inverse of $L_{i}$ with respect to $L_{i+2}$. We have,

\begin{align}
\label{I}
a'_3 &\equiv -nL_{i}^{-1} \hspace{.2cm}(\bmod \hspace{.2cm}L_{i+2}) \notag\\
&\equiv -n (L_{i+2}-L_{i+1})^{-1} \hspace{.2cm} (\bmod \hspace{.2cm}L_{i+2}) \notag\\
&\equiv -n (-L_{i+1})^{-1} \hspace{.2cm} (\bmod \hspace{.2cm}L_{i+2}) \notag\\
&\equiv n L_{i+1}^{-1} \hspace{.2cm} (\bmod \hspace{.2cm}L_{i+2}) \notag\\
\end{align}
Using Casssini's identity, we have
$$L_{i+1}^2-L_iL_{i+2}\equiv (-1)^{i+1}5$$
$$L_{i+1}L_{i+1}\equiv (-1)^{i+1}5(\bmod L_{i+2})$$
$$L_{i+1} \left((-1)^{i+1}L_{i+1}\ 5^{-1}\right) \equiv 1(\bmod L_{i+2})$$
\begin{equation}
\label{II}
L_{i+1}^{-1}\equiv (-1)^{i+1}L_{i+1}\ 5^{-1} (\bmod L_{i+2})
\end{equation}
Using equations \eqref{I} and \eqref{II}, we have 
\begin{align*}
 a'_3 &\equiv (-1)^{i+1} n L_{i+1}\ 5^{-1}\hspace{.2cm} (\bmod \hspace{.2cm} L_{i+2})\\
   &\equiv A''_3 \hspace{.2cm} (\bmod \hspace{.2cm} L_{i+2}). 
\end{align*}
Next, we compute the quantities $c'_1$, $a'_2$ and $b'_3$. Note that 
\begin{align*}
c'_1 &\equiv bc^{-1} \hspace{.2cm} (\bmod \hspace{.2cm}a) \\
&\equiv L_{i+1}L_{i+2}^{-1} \hspace{.2cm}(\bmod \hspace{.2cm}L_i) \\
&\equiv (L_{i+2}-L_i) L_{i+2}^{-1} \hspace{.2cm} (\bmod \hspace{.2cm}L_i) \\
&\equiv L_{i+2} L_{i+2}^{-1} \hspace{.2cm} (\bmod \hspace{.2cm}L_i) \\
&\equiv 1 \hspace{.2cm} (\bmod \hspace{.2cm}L_i). 
\end{align*}

Thus, $c'_1 \equiv 1$ mod $L_i$, and $1 \leq c'_1 \leq L_i$. 
Similarly, 

\begin{align*}
a'_2 &\equiv ca^{-1} \hspace{.2cm} (\bmod \hspace{.2cm}b) \\
&\equiv L_{i+2}L_{i}^{-1} \hspace{.2cm} (\bmod \hspace{.2cm}L_{i+1}) \\
&\equiv (L_{i+1}+L_i) L_{i}^{-1} \hspace{.2cm} (\bmod \hspace{.2cm}L_{i+1}) \\
&\equiv L_{i} L_{i}^{-1} \hspace{.2cm} (\bmod \hspace{.2cm} L_{i+1}) \\
&\equiv 1 \hspace{.2cm} (\bmod \hspace{.2cm}L_{i+1}). 
\end{align*}

Thus, $a'_2 \equiv 1$ mod $L_{i+1}$, and $1 \leq a'_2 \leq L_{i+1}$. 

\begin{align*}
b'_3 &\equiv ab^{-1} \hspace{.2cm} (\bmod \hspace{.2cm}c) \\
&\equiv L_{i}L_{i+1}^{-1} \hspace{.2cm} (\bmod \hspace{.2cm}L_{i+2}) \\
&\equiv (L_{i+2}-L_{i+1}) L_{i+1}^{-1} \hspace{.2cm} (\bmod \hspace{.2cm}L_{i+2}) \\
&\equiv -L_{i+1} L_{i+1}^{-1} \hspace{.2cm} (\bmod \hspace{.2cm}L_{i+2}) \\
&\equiv -1 \hspace{.2cm} (\bmod \hspace{.2cm}L_{i+2}). 
\end{align*}
Thus, $b'_3 \equiv -1$ mod $L_{i+2}$, and $1 \leq b'_3 \leq L_{i+2}$. That is, $b'_3 = L_{i+2} - 1$. 
Putting the values of $b'_1, c'_2, a'_3$, and of $c'_1, a'_2, b'_3$ in the expression for $N_1$, we obtain $N_1 = N_3$. Finally, we solve the summations involved in Binner's formula. Using $b'_1 \leq L_i$, $c'_2 \leq L_{i+1}$, and $a'_3 \leq L_{i+2}$, we have

$$\sum_{i=1}^{b'_1-1} \left\lfloor \frac{ic'_1}{a} \right\rfloor=\sum_{i=1}^{b'_1-1} \left\lfloor \frac{ic'_1}{L_i} \right\rfloor =\sum_{i=1}^{b'_1-1} \left\lfloor \frac{i}{L_i} \right\rfloor=0,$$\\
$$\sum_{i=1}^{c'_2-1} \left\lfloor \frac{ia'_2}{b} \right\rfloor=\sum_{i=1}^{c'_2-1} \left\lfloor \frac{ia'_2}{L_{i+1}} \right\rfloor=\sum_{i=1}^{c'_2-1} \left\lfloor \frac{i}{L_{i+1}} \right\rfloor=0, $$

\begin{align*}
\sum_{j=1}^{a'_3-1} \left\lfloor \frac{jb'_3}{c} \right\rfloor &=\sum_{j=1}^{a'_3-1} \left\lfloor \frac{j(L_{i+2}-1)}{L_{i+2}} \right\rfloor\\
&=\sum_{j=1}^{a'_3-1} \left\lfloor j- \frac{j}{L_{i+2}} \right\rfloor\\     
&=\sum_{j=1}^{a'_3-1}(j-1)\\
&=\frac{(a'_3-2)(a'_3-1)}{2}.
\end{align*}
Putting the values of these summations in equation \eqref{Generall} completes the proof of our Theorem \ref{FibFrobb}. 

\begin{example}   
\normalfont Suppose $i=10$. Then, $L_{i} = 123$, $L_{i+1} = 199$, and $L_{i+2} = 322$. Consider the equation $123x+199y+322z=425896$.\\
By simple calculations we obtain\\
$$B''_1 = 65,\ C''_2 = 74,\  A''_3 = 168\ and\ N_3=-62942409684.$$\\
From \eqref{FibFrobb} The number of non-negative integer solutions of $123x+199y+322z=394072$  is given by:
 $$N(L_i,L_{i+1},L_{i+2};n)=\frac{N_3}{2L_iL_{i+1}L_{i+2}}-2+\frac{(A''_3-1)(A''_3-2)}{2}$$
 Hence,
 $$N(123,199,322;394072)=9532$$
\end{example}
\begin{center}
\textbf{Acknowledgements}
\end{center}
The author is greatful to Dr. Damanvir Singh Binner for his guidance and helpful discussions. In addition, the author is also thankful to the Maths Department at Sant Longowal Institute Of Engineering And Technology, Sangrur, Punjab for providing a good research environment. 

\end{document}